\documentclass[]{amsart}

\usepackage{amsbsy, amscd, amsmath, amssymb, amstext, amsxtra, commath, eucal, float, graphicx, hyperref, latexsym, mathrsfs, mathtools, multicol, tikz, tikz-cd}
\usepackage[utf8]{inputenc}
\usepackage[all]{xy}

\usepackage{lineno}

\newtheorem{theorem}{Theorem}[section]
\newtheorem{lemma}[theorem]{Lemma}
\newtheorem{proposition}[theorem]{Proposition}
\newtheorem{corollary}[theorem]{Corollary}
\theoremstyle{definition}
\newtheorem{definition}[theorem]{Definition}
\newtheorem{example}[theorem]{Example}

\numberwithin{equation}{section}

\newtheorem{question}[theorem]{Question}

\newcommand\restr[2]{\ensuremath{#1\!\!\upharpoonright_{#2}}\,}

\def\CN{\mathop{\mathsf{CN}}\nolimits}
\def\ON{\mathop{\mathsf{ON}}\nolimits}
\def\SC{\mathop{\mathsf{SC}}\nolimits}

\def\cf{\mathop{\rm cf}\nolimits}

\def\CL{\mathop{\rm CL}\nolimits}

\def\int{\mathop{\rm int}\nolimits}

\def\CL{\mathop{\rm CL}\nolimits}
\def\K{\mathop{\mathcal{K}\hspace{0mm}}\nolimits}
\def\F{\mathop{\mathcal{F}\hspace{0mm}}\nolimits}
\def\HH{\mathop{\mathcal{H}\hspace{0mm}}\nolimits}

\def\Pot{\mathop{\mathscr{P}\hspace{0mm}}\nolimits}

\newcommand{\vt}[1]{\langle #1\rangle}

\let\mathcal\mathscr

\tikzcdset{scale cd/.style={every label/.append style={scale=#1},
    cells={nodes={scale=#1}}}}
    
    \newcommand{\stickT}{\setbox255=\hbox{\raise1ex\hbox{$\hspace{0.2pt}\,\bullet\,$}}\mathord{\rlap{\hbox to\wd255{\hss\hbox{$|$}\hss}}\box255}}
\newcommand{\stickS}{\setbox255=\hbox{\raise0.6ex\hbox{$\scriptstyle\bullet$}}\mathord{\rlap{\hbox to\wd255{\hss\hbox{$\scriptstyle|$}\hss}}\box255}}
\newcommand{\stick}{{\mathchoice{\stickT}{\stickT}{\stickS}{\stickS}}}

\begin{document}

\subjclass[2020]{54A25, 54A35, 54B20.}

\keywords{super calibers, calibers, Vietoris hyperspaces.}

%\linenumbers

\title{Super calibers in topological spaces and topological hyperspaces}

\author{Alejandro R\'ios Herrej\'on}
\address {A. Ríos Herrej\'{o}n\\
Departamento de Matem\'aticas, Facultad de Ciencias, Universidad Nacional Aut\'onoma de M\'exico, Circuito ext. s/n, Ciudad Universitaria, C.P. 04510,  M\'exico, CDMX}
\email{chanchito@ciencias.unam.mx}

\begin{abstract}

We study the notion of a super caliber of a topological space, which is closely related to the classical notion of caliber and has appeared in the literature under several different names. We collect and unify several known results and establish new results concerning the collections of super calibers of topological spaces and their hyperspaces. In particular, we investigate the relationship between the super calibers of a space $X$ and those of hyperspaces $\mathcal{H}(X)$ lying between $\CL(X)$ and $\mathcal{F}(X)$. For infinite metrizable spaces, we characterize several cases in which $\SC(X)$ and $\SC(\CL(X))$ differ and establish an independence result over \textsf{ZFC}; see Theorem~\ref{thm_CF_principal}.

\end{abstract}

\maketitle

%%%%%%%%%%%%%%%%%%%%%%%%%%%%%%%%%%%%%%%%%%%%%%%%%%%%%%%%%%%%%%
\section{Introduction}
%%%%%%%%%%%%%%%%%%%%%%%%%%%%%%%%%%%%%%%%%%%%%%%%%%%%%%%%%%%%%%

In recent years, the study of chain conditions and, in particular, caliber theory in topological spaces has received considerable attention from the mathematical community.

A notion closely related to the property of a space having a given caliber has also been extensively studied in the literature, particularly in the case of the cardinal $\omega_1$, and under several different names.

For example, the property $C(\omega_1)$, or, equivalently, the property of belonging to the class $C(\omega_1)$, has been studied in~\cite{dowjuh2025}, \cite{jss2025}, \cite{linxu2026}, and \cite{rios2023}. For $\kappa=\omega_1$, this property coincides with the property $\mathrm{suHG}$ studied by Hart and Kunen in~\cite{harkun2020}, as well as with the property $(C')$ studied by Sakai in~\cite{sakai2012}.

More generally, if $\kappa$ is an infinite cardinal, we say that a topological space $X$ has \emph{super caliber} $\kappa$ if, for every $\{x_\alpha : \alpha<\kappa\} \subseteq X$ and $\{U_\alpha : \alpha<\kappa\} \subseteq \tau_X$ satisfying $x_\alpha \in U_\alpha$ for each $\alpha<\kappa$, there exists a set $J \in [\kappa]^{\kappa}$ such that
\[
\textstyle
\{x_\alpha : \alpha \in J\} \subseteq \bigcap_{\alpha \in J} U_\alpha,
\]
or, equivalently, $x_\alpha \in U_\beta$ for every $\alpha,\beta \in J$. Thus, in the case $\kappa = \omega_1$, having the property $C(\omega_1)$, the property $(C')$, or the property $\mathrm{suHG}$ is equivalent to having super caliber $\omega_1$.

The aim of this paper is twofold. On the one hand, in the first part of the paper, we provide a brief and unified account of basic results known in the literature, as well as of some more substantial results, with the purpose of bringing them together under a common terminology, namely, the notion of a super caliber of a topological space. We also contribute to the study of what may be viewed as the theory of super calibers in topological spaces by presenting several examples of spaces and their collections of super calibers, as well as by comparing classical results from caliber theory with their counterparts in the theory of super calibers.

On the other hand, \cite{riotam2026} provides a general study of the relationship between the calibers, precalibers, and weak precalibers of a topological space $X$ and those of its hyperspaces $\mathcal{H}(X)$ lying between $\CL(X)$, the hyperspace of nonempty closed subsets of $X$, and $\mathcal{F}(X)$, the subspace consisting of finite subsets.

With this in mind, another part of this work is devoted to contributing to hyperspace theory by analyzing the relationship between the super calibers of a topological space $X$ and those of its hyperspaces $\mathcal{H}(X)$. In marked contrast with the results obtained in~\cite{riotam2026}, we exhibit several notable differences, ranging from examples in which the two collections of super calibers do not coincide to independence results concerning hyperspaces over infinite metrizable spaces.

More specifically, Section~\ref{secc_pre} is devoted to introducing notation that will be used in what follows. In Section~\ref{secc_SC}, we review several known results from the literature concerning super calibers and develop some basic aspects of the concept. We then turn, in Section~\ref{secc_TH}, to the study of hyperspaces in the context of super calibers. Finally, in Section~\ref{secc_THM}, we analyze hyperspaces over metrizable spaces and establish an independence result over \textsf{ZFC}.

%%%%%%%%%%%%%%%%%%%%%%%%%%%%%%%%%%%%%%%%%%%%%%%%%%%%%%%%%%%%
\section{Preliminaries}\label{secc_pre}
%%%%%%%%%%%%%%%%%%%%%%%%%%%%%%%%%%%%%%%%%%%%%%%%%%%%%%%%%%%%

The purpose of this brief section is to establish the terminology and notation that we use systematically throughout the paper. Unless otherwise indicated, we follow the standard conventions in the literature.

We use the symbol $\omega$ to denote both the first infinite ordinal and the first infinite cardinal. We also define $\mathbb{N}:=\omega\setminus\{0\}$. The symbol $\mathfrak{c}$ denotes the cardinality of the continuum, namely, the cardinal $2^\omega$. 

On the other hand, if $X$ is a set, we denote its \emph{power set} by $\mathcal{P}(X)$. Also, if $\kappa$ is a cardinal number, we define
\[
[X]^{\kappa} := \{A \subseteq X : |A| = \kappa\} \quad \text{and} \quad [X]^{<\kappa} := \{A \subseteq X : |A| < \kappa\}.
\]

Let $X$ be a topological space. We denote by $\tau_X$ the topology on $X$, and by $\tau_X^+$ the set $\tau_X\setminus\{\emptyset\}$ of all nonempty open subsets of $X$.

We say that an infinite cardinal $\kappa$ is a \emph{caliber} for $X$ if, for every family $\{U_\alpha : \alpha<\kappa\} \subseteq \tau_X^+$, there exists a set $J \in [\kappa]^\kappa$ such that the family $\{U_\alpha : \alpha\in J\}$ has nonempty intersection.

Finally, we use the symbols $\ON$ and $\CN$ to denote, respectively, the class of all ordinals and the class of all infinite cardinals.

%%%%%%%%%%%%%%%%%%%%%%%%%%%%%%%%%%%%%%%%%%%%%%%%%%%%%%%%%%%%
\section{Super calibers}\label{secc_SC}
%%%%%%%%%%%%%%%%%%%%%%%%%%%%%%%%%%%%%%%%%%%%%%%%%%%%%%%%%%%%

As mentioned in the introduction, \cite{rios2023} studies the property $C(\kappa)$, where $\kappa$ is an infinite cardinal, while \cite{linxu2026} studies the same property when $\cf(\kappa)>\omega$. The papers \cite{dowjuh2025} and \cite{jss2025} focus in particular on the case $\kappa=\omega_1$. The property $C(\omega_1)$ coincides with the property $(C')$ studied by Sakai in~\cite{sakai2012} and with the property $\mathrm{suHG}$ studied by Hart and Kunen in~\cite{harkun2020}.

Since the property $C(\kappa)$ is a substantial strengthening of the notion of having caliber $\kappa$, we shall refer to it as follows:

\begin{definition}\label{def_CF} Let $X$ be a topological space and let $\kappa$ be an infinite cardinal. We say that $X$ has \emph{super caliber} $\kappa$ if, for every $\{x_\alpha : \alpha<\kappa\} \subseteq X$ and $\{U_\alpha : \alpha<\kappa\} \subseteq \tau_X$ satisfying $x_\alpha \in U_\alpha$ for each $\alpha<\kappa$, there exists a set $J \in [\kappa]^{\kappa}$ such that $\{x_\alpha : \alpha \in J\} \subseteq \bigcap_{\alpha \in J} U_\alpha$, or, equivalently, $x_\alpha \in U_\beta$ for every $\alpha,\beta \in J$.

\end{definition}

It is worth briefly noting that, originally, the proposed term was ``strong caliber''; however, in~\cite{jssv2023}, this terminology is used to designate a cardinal $\kappa$ with the property that, for every $\mathcal{U} \in [\tau_X]^\kappa$, there exists $\mathcal{V} \in [\mathcal{U}]^{\kappa}$ such that the intersection $\bigcap\mathcal{V}$ has nonempty interior. Clearly, this concept does not coincide with the one established in Definition~\ref{def_CF}.

With the change in terminology above, some of the results developed in \cite{dowjuh2025}, \cite{harkun2020}, \cite{jss2025}, and \cite{rios2023} are collected in Proposition~\ref{prop_CF_basico}. In particular, item~(9) was obtained in \cite{jss2025}, item~(10) in \cite{dowjuh2025}, item~(11) in \cite{harkun2020}, and the remaining items in \cite{rios2023}.

\begin{proposition}\label{prop_CF_basico} The following statements hold for every topological space $X$ and every infinite cardinal $\kappa$.

\begin{enumerate}
\item If $X$ has super caliber $\kappa$ and $Y$ is a subspace of $X$, then $Y$ has super caliber $\kappa$. In particular, every subspace of $X$ has caliber $\kappa$.
\item If $X$ has super caliber $\kappa$, then $X$ has super caliber $\cf(\kappa)$.
\item If $X$ has super caliber $\kappa$, then every subspace $Y$ of $X$ satisfies $d(Y) < \cf(\kappa)$. In particular, $hd(X) \leq \cf(\kappa)$.
\item If $\cf(\kappa)>nw(X)$, then $X$ has super caliber $\kappa$.
\item If $X$ has super caliber $\kappa$ and $Y$ is a continuous image of $X$, then $Y$ has super caliber $\kappa$.
\item If $n\in \mathbb{N}$ and $\{X_k : k<n\}$ is a family of topological spaces, each with super caliber $\kappa$, then $\prod_{k<n} X_k$ has super caliber $\kappa$.
\item If $\lambda$ is a positive cardinal, $\cf(\kappa)>\lambda$, and $\{X_\alpha : \alpha<\lambda\}$ is a family of subspaces of $X$, each with super caliber $\kappa$, then $\bigcup_{\alpha<\lambda} X_\alpha$ has super caliber $\kappa$.
\item If $\lambda$ is a positive cardinal and $\{X_\alpha : \alpha<\lambda\}$ is a family of topological spaces, then $\bigoplus_{\alpha<\lambda} X_\alpha$ has super caliber $\kappa$ if and only if $\cf(\kappa)>\lambda$ and each $X_\alpha$ has super caliber $\kappa$.
\item Under the combinatorial principle $\stick_{\text{SUPER}}$, if $X$ is $T_3$ and has super caliber $\omega_1$, then $nw(X) = \omega$.
\item There exists a Hausdorff space $Z$ that has super caliber $\omega_1$ and satisfies $nw(X)>\omega$.
\item It is consistent with \textsf{ZFC} to have \textsf{MA} and $\mathfrak{c} = \aleph_2$ and a first countable space $X$ having super caliber $\omega_1$ with $|X| = w(X) = nw(X) = \aleph_2$.

\end{enumerate}

\end{proposition}

It is worth noting that there are also important connections between the theory of super calibers and the theory of Pixley--Roy hyperspaces. Although we do not include these results in Proposition~\ref{prop_CF_basico}, since it is devoted exclusively to topological spaces in general, we superly encourage interested readers to consult \cite{linxu2026}, \cite{rios2023}, and \cite{sakai2012}.

Given a topological space $X$, in this text we call the cardinal invariant
\[
s\check{s}(X) := \min\{\kappa \in \CN : \text{$X$ has super caliber $\kappa^+$}
\]
the \emph{super \v{S}anin number}. Naturally, this concept is related to the traditional \emph{\v{S}anin number}:
\[
\check{s}(X) := \min\{\kappa \in \CN : \text{$X$ has caliber $\kappa^+$}\}.
\]
Clearly, by items (1) and (4) of Proposition~\ref{prop_CF_basico}, we have $\check{s}(X) \leq s\check{s}(X) \leq nw(X)$. We also consider the collection
\[
\SC(X) := \{\kappa \in \CN : \text{$\kappa$ is a super caliber for $X$}\}.
\]

Proposition~\ref{prop_CF_metrizable} will be fundamental to several of the results presented in this section:

\begin{proposition}\label{prop_CF_metrizable} If $X$ is an infinite metrizable space, then
\[
\SC(X) = \{\kappa \in \CN : \cf(\kappa)>nw(X)\}.
\]
In particular, $s\check{s}(X) = nw(X)$.

\end{proposition}

\begin{proof} First, the right-to-left inclusion follows from Proposition~\ref{prop_CF_basico}(4). Conversely, let $\{U_\alpha : \alpha < nw(X)\}$ be a cellular family of cardinality $nw(X)$ (see \cite{hodel1984}), and suppose that $\kappa\in \CN$ satisfies $\cf(\kappa) \leq nw(X)$. Then the collection $\{U_\alpha : \alpha < \cf(\kappa)\}$ witnesses that $X$ does not have super caliber $\cf(\kappa)$ and, consequently, that it does not have super caliber $\kappa$, by Proposition~\ref{prop_CF_basico}(2).
\end{proof}

A natural question is whether the equality established in Proposition~\ref{prop_CF_metrizable} remains valid when ``metrizable'' is replaced by ``submetrizable'' (that is, a topological space whose topology contains a metrizable topology). The answer is negative:

\begin{example}\label{ej_CF_submetrizable} There exists a submetrizable space $X$ such that
\[
\omega_1 \in \SC(X) \setminus \{\kappa \in \CN : \cf(\kappa) > nw(X)\}.
\]

\medskip

In \cite[Theorem~2, p.~2]{dowjuh2025}, an example is constructed of a submetrizable space $X$, obtained by refining the usual Euclidean topology on the real line, with the property that $\omega_1$ is a super caliber for $X$ and, at the same time, the net weight of $X$ is uncountable. Clearly, $X$ has the desired properties.

\end{example}

A topological space $X$ is \emph{weakly separated} if there exists a family $\{U_x : x \in X\} \subseteq \tau_X$ satisfying the following conditions: $x \in U_x$ for every $x \in X$, and if $x,y \in X$ are distinct, then $x \notin U_y$ or $y \notin U_x$. This class of spaces was introduced by Tkachenko in \cite{tkachenko1978} and has been extensively studied in \cite{bell1979} and \cite{rios2023}.

For example, every countable $T_1$ space, every right-separated space (in particular, every discrete space), every left-separated space, the Sorgenfrey line, and the Pixley--Roy hyperspace $\F[X]$ over any $T_1$ space are weakly separated (see \cite{rios2023}).

In the context of the present work, weakly separated spaces exhibit a property similar to that of metrizable spaces in the sense of Proposition~\ref{prop_CF_metrizable}. This similarity is reflected in Proposition~\ref{prop_CF_deb_sep}, and its proof requires only Proposition~\ref{prop_CF_basico}(4) together with three facts: if $X$ is a weakly separated space, then every subspace of $X$ is weakly separated, $|X| = nw(X)$ (see \cite{bell1979}), and $|X|$ is not a super caliber for $X$.

\begin{proposition}\label{prop_CF_deb_sep} If $X$ is an infinite weakly separated space, then
\[
\SC(X) = \{\kappa \in \CN : \cf(\kappa)>nw(X)\}.
\]

\end{proposition}

The \emph{weak separation number} of $X$ is the cardinal number
$$
R(X) := \sup\{|Y| : Y \text{ is a weakly separated subspace of } X\} + \omega.
$$
If $hd(X)$ and $hL(X)$ denote, respectively, the \emph{hereditary density} and the \emph{hereditary Lindelöf degree} of $X$, it is known that
\[
hd(X)\cdot hL(X) \leq R(X) \leq nw(X)
\]
(see \cite{rios2023}). Consequently, if $\varphi^*(X) := \sup\{\varphi(X^n) : n \in \mathbb{N}\}$, where $\varphi$ is a topological cardinal function (in particular, $nw^*(X) = nw(X)$; see \cite{hodel1984}), it follows that
\[
hd^*(X)\cdot hL^*(X) \leq R^*(X) \leq nw(X).
\]

This brief introduction to weakly separated spaces provides the necessary context for the following result:

\begin{proposition}\label{prop_CF_FCT} If $X$ is a topological space, $n \in \mathbb{N}$, $Y$ is a metrizable or weakly separated subspace of $X^n$, and $\kappa$ is a super caliber for $X$, then $nw(Y) < \cf(\kappa)$. Consequently,
\[
R^*(X) \leq s\check{s}(X) \leq nw(X).
\]
\end{proposition}

\begin{proof}
First, observe that the power $X^n$ has super caliber $\kappa$, by Proposition~\ref{prop_CF_basico}(6). Since this property is hereditary, $Y$ also has super caliber $\kappa$. Therefore, if $Y$ is infinite, Proposition~\ref{prop_CF_metrizable} or Proposition~\ref{prop_CF_deb_sep} guarantees that $\cf(\kappa) > nw(Y)$. In particular, $R^*(X) \leq s\check{s}(X)$.
\end{proof}

\begin{corollary}\label{cor_CF_potencias_hd_hL} If a topological space $X$ has super caliber $\omega_1$, then all finite powers of $X$ are hereditarily separable and hereditarily Lindelöf.

\end{corollary}

Corollary~\ref{cor_CF_subespacio_nw} follows from Proposition~\ref{prop_CF_basico}(4) and Proposition~\ref{prop_CF_FCT}:

\begin{corollary}\label{cor_CF_subespacio_nw} If a topological space $X$ has a metrizable or weakly separated subspace $Y$ such that $nw(Y) = nw(X)$, then
\[
\SC(X) = \{\kappa \in \CN : \cf(\kappa) > nw(X)\}.
\]

\end{corollary}

For example, the Sorgenfrey plane and the Moore--Niemytzki plane are separable Tychonoff spaces whose collections of super calibers are $\{\mu \in \CN : \cf(\mu) > \mathfrak{c}\}$, since both contain a discrete subspace of cardinality $\mathfrak{c}$. %The double date also.

On the other hand, as usual, the symbol $\beta \omega$ denotes the \emph{Stone--\v{C}ech compactification} of the discrete space of cardinality $\omega$, while $\omega^*$ denotes the remainder of this compactification. It is well known that $nw(\beta \omega) = \mathfrak{c} = nw(\omega^*)$, and that $\omega^*$ admits a cellular family of cardinality $\mathfrak{c}$; in particular, $\omega^*$ contains a discrete subspace of cardinality $\mathfrak{c}$. Therefore, Corollary~\ref{cor_CF_subespacio_nw} implies that
\[
\SC(\beta \omega) = \{\kappa \in \CN : \cf(\kappa) > \mathfrak{c}\} = \SC(\omega^*).
\]

It is also worth noting that there is an important class of compact Hausdorff spaces to which Corollary~\ref{cor_CF_subespacio_nw} applies. We say that a compact $T_2$ space is \emph{polyadic} if it is a continuous image of a power of the one-point compactification of an infinite discrete space. If $X$ is a polyadic compact space, then $X$ contains a discrete subspace of cardinality $w(X)$ (see \cite{gerlits1978}). This, together with Corollary~\ref{cor_CF_subespacio_nw} and the fact that $nw(X) = w(X)$ whenever $X$ is compact and $T_2$ (see \cite{hodel1984}), yields the following corollary:

\begin{corollary} If $X$ is an infinite polyadic compact space, then
\[
\SC(X) = \{\kappa \in \CN : \cf(\kappa)>nw(X)\}.
\]

\end{corollary}

The class of polyadic compact spaces is also broad. To see this, observe that if $\kappa \geq \omega$ and $\alpha D(\kappa)$ denotes the Alexandroff compactification of the discrete space of cardinality $\kappa$, then the function $f: \alpha D(\kappa) \to D(2)$ defined by $f(0) = 0$ and $f[\alpha D(\kappa) \setminus \{0\}] \subseteq \{1\}$ is continuous and surjective. Consequently, every dyadic compact space (that is, every continuous $T_2$ image of a Cantor cube) is polyadic; in particular, every compact topological group belongs to this class (see \cite{shakhmatov1994}).

Returning to our discussion, it is well known that if $X$ is a topological space, $D$ is a dense subspace of $X$, and $\kappa$ is a caliber for $D$, then $\kappa$ is also a caliber for $X$. However, the converse does not hold in general. For example, if $\kappa$ is a cardinal number with $\cf(\kappa)>\omega$, then in the Cantor cube of weight $\kappa$, $D(2)^{\kappa}$, the $\Sigma$-product
\[
\Sigma(\kappa) := \{x \in D(2)^{\kappa} : |x^{-1}\{1\}| \leq \omega\},
\]
is a dense subspace of $D(2)^{\kappa}$ that does not have caliber $\kappa$, even though $D(2)^{\kappa}$ does have caliber $\kappa$ (see \cite{shelah1977}).

For the corresponding question concerning super calibers, it is immediate that, since the property is hereditary, every dense subspace of a space with super caliber $\kappa$ also has super caliber $\kappa$. A natural question is whether the converse holds. We present a result in this direction below. Recall that, if $\kappa$ is a cardinal, the symbol $\beth_2(\kappa)$ denotes the cardinal $2^{2^\kappa}$.

\begin{proposition}\label{prop_CF_denso} The following statements hold for every infinite cardinal $\kappa$.

\begin{enumerate}
\item There exists a compact Hausdorff space $X_\kappa$ such that $d(X_\kappa)=\kappa$ and
$$
\SC(X_\kappa)=\{\mu\in\CN : \cf(\mu)>\kappa\}.
$$

\item For every cardinal $\lambda$ with $\kappa\leq\lambda\leq\beth_2(\kappa)$, there exists a Urysohn space $X(\kappa,\lambda)$ such that $d(X(\kappa,\lambda))=\kappa$ and $$\SC(X(\kappa,\lambda))=\{\mu\in\CN : \cf(\mu)>\lambda\}.$$

\end{enumerate}

\end{proposition}

\begin{proof} To begin, let $X_\kappa$ be the ordinal space $[0,\kappa]$ (that is, the set $\kappa \cup \{\kappa\}$ equipped with the order topology). Then $X_\kappa$ is compact and $T_2$. Moreover, $\{\alpha+1 : \alpha<\kappa\}$ is a discrete and dense subspace of $X_\kappa$, consisting of isolated points and having cardinality $\kappa=d(X_\kappa)=nw(X_\kappa)$. Therefore, by Corollary~\ref{cor_CF_subespacio_nw},
\[
\SC(X_\kappa)=\{\mu\in\CN : \cf(\mu)>\kappa\}.
\]

Now, let $\lambda$ be a cardinal such that $\kappa \leq \lambda \leq \beth_2(\kappa)$. The space $X(\kappa,\lambda)$ is a subspace of the Kat\v{e}tov extension of the discrete space of cardinality $\kappa$. More precisely, let
\[
\mathbb{F} := \{\mathcal{F} : \text{$\mathcal{F}$ is a free ultrafilter on $\kappa$}\},
\]
let $\mathbb{F}_\lambda$ be a subset of $\mathbb{F}$ of cardinality $\lambda$, and let $X(\kappa,\lambda)$ be the set $\kappa \cup \mathbb{F}_\lambda$ equipped with the following topology: each point of $\kappa$ is isolated, and, for each $\mathcal{F} \in \mathbb{F}_\lambda$, the family $\{F \cup \{\mathcal{F}\} : F \in \mathcal{F}\}$ is a local base for $\mathcal{F}$ in $X(\kappa,\lambda)$.

The space $X(\kappa,\lambda)$ is Urysohn, $\kappa$ is a dense subspace of $X(\kappa,\lambda)$ consisting of isolated points (in particular, $d(X(\kappa,\lambda))=\kappa$), and $\mathbb{F}_\lambda$ is a closed discrete subspace of $X(\kappa,\lambda)$ of cardinality $\lambda$.

Finally, since $\lambda = |\mathbb{F}_\lambda| = nw(\mathbb{F}_\lambda) \leq nw(X(\kappa,\lambda)) \leq |X(\kappa,\lambda)| = \lambda$, it follows that $\mathbb{F}_\lambda$ is a discrete subspace of $X(\kappa,\lambda)$ of cardinality $nw(X(\kappa,\lambda))$. Hence, by Corollary~\ref{cor_CF_subespacio_nw},
\[
\SC(X(\kappa,\lambda))=\{\mu\in\CN : \cf(\mu)>\lambda\}.
\]
\end{proof}

Proposition~\ref{prop_CF_denso} provides examples showing that super calibers of dense subspaces need not transfer to the entire space:

\begin{corollary}\label{cor_CF_denso} For every infinite cardinal $\kappa$, the space $X(\kappa,\beth_2(\kappa))$ described in Proposition~\ref{prop_CF_denso} has a dense subspace with super caliber $\kappa^+$, but $X(\kappa,\beth_2(\kappa))$ does not have super caliber $\kappa^+$ or $\kappa^{++}$.

\end{corollary}

Before continuing, it is worth noting that, in the particular case of $\omega_1$, it is not difficult to provide examples of spaces with various interesting properties in which $\omega_1$ does not transfer as a super caliber from a dense subspace:

\begin{example} The Sorgenfrey line, $\mathbb{S}$, is hereditarily separable and hereditarily Lindelöf, $T_5$, and has net weight $\mathfrak{c}$. Since $\mathbb{S}$ is also weakly separated, Proposition~\ref{prop_CF_deb_sep} yields $\SC(\mathbb{S}) = \{\kappa \in \CN : \cf(\kappa)>\mathfrak{c}\}$. On the other hand, the rational numbers $\mathbb{Q}$ form a countable, dense, and weakly separated subspace of $\mathbb{S}$; in particular, the same result gives $\SC(\mathbb{Q}) = \{\kappa \in \CN : \cf(\kappa)>\omega\}$. Consequently, $\omega_1 \in \SC(\mathbb{Q}) \setminus \SC(\mathbb{S})$.

Similarly, if $\mathbb{M}$ denotes the Moore--Niemytzki plane, then, since $\mathbb{Q}^2$ is a dense subspace of both $\mathbb{M}$ and $\mathbb{S}^2$, we also have $\omega_1 \in \SC(\mathbb{Q}^2) \setminus \SC(\mathbb{S}^2)$ and $\omega_1 \in \SC(\mathbb{Q}^2) \setminus \SC(\mathbb{M})$, by the remarks in the paragraph following Corollary~\ref{cor_CF_subespacio_nw}.

\end{example}

Regarding the preservation of super calibers under topological products, two aspects are worth noting. First, Proposition~\ref{prop_CF_basico}(6) states that having super caliber $\kappa$, for a given infinite cardinal $\kappa$, is finitely productive. Second, it follows from the results presented in \cite{rios2022} and \cite{shelah1977} that if $X$ is a topological space with caliber, precaliber, or weak precaliber $\kappa$, where $\cf(\kappa)>\omega$, and $\lambda$ is a cardinal, then the topological power $X^\lambda$ also has the corresponding chain condition.

One might expect, at least, that uncountable regular super calibers would be preserved under topological products. However, it is easy to see that this is not the case, even under minimal assumptions:

\begin{lemma}\label{lema_CF_producto} Let $\lambda$ be an infinite cardinal and let $\{X_\alpha : \alpha < \lambda\}$ be a family of topological spaces. Suppose, moreover, that each space $X_\alpha$ has two disjoint, nonempty open subsets. If $\kappa$ is a cardinal such that $\omega \leq \kappa \leq \lambda$, then $\prod_{\alpha < \lambda} X_\alpha$ does not have super caliber $\kappa$.

\end{lemma}

\begin{proof} First, observe that, under the stated hypotheses, the Cantor cube $D(2)^\lambda$ is homeomorphic to a subspace of the product $\prod_{\alpha < \lambda} X_\alpha$. Since having super caliber $\kappa$ is a hereditary property (see Proposition~\ref{prop_CF_basico}(1)), it suffices to show that $D(2)^\lambda$ does not have super caliber $\kappa$. To this end, we will exhibit a subspace of $D(2)^\lambda$ that does not have super caliber $\kappa$.

For each $\alpha < \kappa$, define the function $x_\alpha : \lambda \to D(2)$ by
\[
x_\alpha(\beta) :=
\begin{cases} 
1, & \text{if } \alpha = \beta, \\
0, & \text{if } \alpha \neq \beta.
\end{cases}
\]
Thus, $\{x_\alpha : \alpha < \kappa\}$ is a discrete subspace of the $\sigma$-product $\sigma(\lambda) := \{x \in D(2)^{\lambda} : |x^{-1}\{1\}| < \omega\}$. Consequently, if $\sigma(\lambda)$ had super caliber $\kappa$, Proposition~\ref{prop_CF_FCT} would imply that $\kappa < \cf(\kappa)$, which is impossible. Therefore, $\prod_{\alpha < \lambda} X_\alpha$ does not have super caliber $\kappa$.
\end{proof}

Based on Lemma~\ref{lema_CF_producto}, it is possible to completely characterize the collection of super calibers for a broad class of topological products:

\begin{proposition}\label{prop_CF_producto} Let $\lambda$ be an infinite cardinal, and let $\{X_\alpha : \alpha < \lambda\}$ be a family of topological spaces. Suppose, moreover, that each space $X_\alpha$ has two disjoint, nonempty open subsets. If $nw(X_\alpha) \leq \lambda$ for every $\alpha < \lambda$, then
\[
\textstyle \SC\big( \prod_{\alpha < \lambda} X_\alpha \big) = \{\mu \in \CN : \cf(\mu) > \lambda\}.
\]

\end{proposition}

\begin{proof} Let $\mu \in \CN$. On the one hand, if $\cf(\mu) \leq \lambda$, Lemma~\ref{lema_CF_producto} shows that $\cf(\mu)$ is not a super caliber for $\prod_{\alpha < \lambda} X_\alpha$ and, consequently, neither is $\mu$, by Proposition~\ref{prop_CF_basico}(2). On the other hand, the fact that $nw(\prod_{\alpha < \lambda} X_\alpha) \leq \lambda$, together with Proposition~\ref{prop_CF_basico}(4), implies that $\mu$ is a super caliber for $\prod_{\alpha < \lambda} X_\alpha$ whenever $\cf(\mu) > \lambda$.
\end{proof}

%%%%%%%%%%%%%%%%%%%%%%%%%%%%%%%%%%%%%%%%%%%%%%%%%%%%%%%%%%%%%%
\section{Topological hyperspaces}\label{secc_TH}
%%%%%%%%%%%%%%%%%%%%%%%%%%%%%%%%%%%%%%%%%%%%%%%%%%%%%%%%%%%%%%

Let $X$ be a topological space. Denote by $\CL(X)$ the collection of all nonempty closed subsets of $X$, and by $\F(X)$ the subcollection of $\CL(X)$ consisting of all finite subsets. A \emph{hyperspace} is any topological space whose underlying set is $\CL(X)$, or a subspace thereof, equipped with a topology for which the natural map $x\mapsto\overline{\{x\}}$ is a topological embedding. Consequently, whenever $X$ satisfies the $T_0$ axiom, the original space can be identified with a subspace of $\CL(X)$.

If $\mathcal{U}$ is a family of subsets of $X$, define
\[
\textstyle
\vt{\mathcal{U}}
:=
\big\{
A\in\CL(X):
A\subseteq\bigcup\mathcal{U}
\ \wedge\
\forall U\in\mathcal{U}\,(A\cap U\neq\emptyset)
\big\}.
\]
The \emph{Vietoris topology} on $\CL(X)$ is the topology generated by the basis
\[
\textstyle
\big\{
\vt{\mathcal{U}}: \mathcal{U} \in [\tau_X]^{<\omega}
\big\}.
\]

For notational convenience, if $n\in\mathbb{N}$ and $U_1,\dots,U_n\in\tau_X$, we write $\vt{U_1,\dots,U_n}$ instead of $\vt{\{U_1,\dots,U_n\}}$. Thus, for example, for any $U,V\in\tau_X$, we have
\[
\textstyle
\vt{U,V}
=
\big\{
A\in\CL(X):
A\subseteq U\cup V
\ \wedge\
A\cap U\neq\emptyset
\ \wedge\
A\cap V\neq\emptyset
\big\}.
\]

Throughout this section, we assume that $X$ is an infinite $T_1$ topological space. Moreover, the hyperspace $\CL(X)$ and each of its subspaces are equipped with the Vietoris topology.

In \cite{riotam2026}, a thorough and detailed study is carried out of the relation between the calibers, precalibers, and weak precalibers of a topological space and those of its associated Vietoris hyperspaces. To continue this line of investigation in the present work, it is natural to pose the following question:

\begin{question}\label{Q_calibre_fuerte} Let $\kappa$ be an infinite cardinal, and let $\HH(X)$ be a hyperspace satisfying $\F(X) \subseteq \HH(X) \subseteq \CL(X)$. Is it true that, if $\HH(X)$ has super caliber $\kappa$, then $X$ also has super caliber $\kappa$? What about the converse?

\end{question}

To begin, since having super caliber $\kappa$ is a hereditary property, Proposition~\ref{prop_CF_basico}(1) immediately implies that if the hyperspace $\HH(X)$ has super caliber for a fixed infinite cardinal $\kappa$, then $X$ also has super caliber $\kappa$, since $X$ is homeomorphic to a subspace of $\HH(X)$.

The second part of Question~\ref{Q_calibre_fuerte} admits a first step toward a partial answer in the following result:

\begin{proposition}\label{prop_CF_F(X)} Let $\kappa$ be a cardinal with $\cf(\kappa)>\omega$. If $\kappa$ is a super caliber for $X$, then $\kappa$ is a super caliber for $\F(X)$.

\end{proposition}

\begin{proof} Let $\{F_\alpha : \alpha < \kappa\} \subseteq \F(X)$ and $\{U_\alpha : \alpha < \kappa\} \subseteq \tau_{\F(X)}$ be such that $F_\alpha \in U_\alpha$ for every $\alpha < \kappa$. Then, for each $\alpha < \kappa$, there exist $n_\alpha \in \mathbb{N}$ and $\{U(\alpha,1),\ldots,U(\alpha,n_\alpha)\} \subseteq \tau_X$ such that
$F_\alpha \in \vt{U(\alpha,1),\ldots,U(\alpha,n_\alpha)}_{\F} \subseteq U_\alpha$. Since the function $\kappa \to \mathbb{N} \times \mathbb{N}$ given by $\alpha \mapsto (|F_\alpha|, n_\alpha)$ has a fiber of cardinality $\kappa$, there exist $I \in [\kappa]^\kappa$ and $m,n \in \mathbb{N}$ such that $|F_\alpha| = m$ and $n_\alpha = n$ for every $\alpha \in I$.

Now, for each $\alpha \in I$, let $\{x(\alpha,1), \ldots,x(\alpha,m)\}$ be an enumeration without repetitions of $F_\alpha$, and let $e_\alpha : \{1,\ldots,m\} \to \{1,\ldots,n\}$ be a surjective function such that $x(\alpha,i) \in U(\alpha,e_\alpha(i))$ for every $i \in \{1,\ldots,m\}$. Next, for each $\alpha \in I$, define
$$
x_\alpha := (x(\alpha,1), \ldots, x(\alpha,m))
\quad \text{and} \quad
V_\alpha := U(\alpha,e_\alpha(1)) \times \cdots \times U(\alpha,e_\alpha(m)).
$$

Observe that $\{V_\alpha : \alpha \in I\}$ is a family of open subsets of $X^m$ satisfying $x_\alpha \in V_\alpha$ for every $\alpha \in I$.

Since $X^m$ has super caliber $\kappa$ (see Proposition~\ref{prop_CF_basico}(6)), there exists $J \in [I]^\kappa$ such that $x_\alpha \in V_\beta$ for all $\alpha,\beta \in J$. Finally, if $\alpha,\beta \in J$, then $F_\alpha \in U_\beta$, since
\begin{align*}
\{x(\alpha,1), \ldots, x(\alpha,m)\}
&\in \vt{U(\beta,e_\beta(1)), \ldots, U(\beta,e_\beta(m))}_{\F} \\
&= \vt{U(\beta,1), \ldots, U(\beta,n)}_{\F}
\subseteq U_\beta.
\end{align*}
Thus, $\F(X)$ has super caliber $\kappa$.
\end{proof}

Since every (super) caliber of an infinite $T_2$ space must have uncountable cofinality, Proposition~\ref{prop_CF_F(X)} yields the following result:

\begin{corollary}\label{cor_CF_F(X)} If $X$ is a Hausdorff space, then
$$
\SC(X) = \SC(\F(X)).
$$

\end{corollary}

Let $\K(X)$ denote the subspace of $\CL(X)$ consisting of the nonempty compact closed subsets of $X$. The following natural question then arises: does Corollary~\ref{cor_CF_F(X)} remain valid if the hyperspace $\F(X)$ is replaced by the hyperspace $\K(X)$? The purpose of the following paragraphs is to introduce the concepts needed to give a negative answer.

According to \cite{mpp2018}, given a Hausdorff space $X$, a set $S \in [X]^\omega$ is called a \emph{nontrivial convergent sequence} if there exists $x \in X$ such that $S \setminus U \in [X]^{<\omega}$ for every $U \in \tau_X$ with $x \in U$. In this case, we say that $x$ is the \emph{limit} of $S$, that $S$ \emph{converges} to $x$, and we write $S \to x$.

With this terminology, we define the hyperspace
$$
\mathcal{S}_c(X) := \{S \in \mathcal{K}(X) : \text{$S$ is a nontrivial convergent sequence in $X$}\}.
$$
We regard this set as a topological space by equipping it with the subspace topology inherited from $\mathcal{K}(X)$. Moreover, for each $x \in X$, define $\mathcal{S}_c(X,x) := \{S \in \mathcal{S}_c(X) : S \to x\}$, and let $L_X := \{x \in X : \mathcal{S}_c(X,x) \neq \emptyset\}$.

Since, by Proposition~\ref{prop_CF_basico}(1), having super caliber $\omega_1$ is a hereditary property, to show that super caliber $\omega_1$ does not transfer from $X$ to the hyperspace $\K(X)$, it suffices to find a space $X$ with super caliber $\omega_1$ such that $\mathcal{S}_c(X)$\footnote{In the context of this work, and in order to consider a hyperspace containing $\F(X)$, it suffices to work with the union $\F(X) \cup \mathcal{S}_c(X)$, viewed as a subspace of $\K(X)$.} does not have this property:

\begin{example}\label{ej_CF_no_sube_Sc(X)} There exists a Hausdorff space $X$ such that \[
\omega_1 \in \SC(X) \setminus \SC(\mathcal{S}_c(X)).
\]

\medskip

In \cite[Example~8.4, p.~440]{mpp2018}, a Hausdorff space $X$ is constructed such that $nw(X)=\omega$ and $\mathcal{S}_c(X)$ has a discrete subspace of cardinality $\mathfrak{c}$. Therefore, Proposition~\ref{prop_CF_basico}(4) and Proposition~\ref{prop_CF_FCT} imply, respectively, that $X$ has super caliber $\omega_1$, whereas $\mathcal{S}_c(X)$ does not have super caliber $\omega_1$.

\end{example}

In light of Example~\ref{ej_CF_no_sube_Sc(X)}, it is not true in general that the super calibers of $X$ transfer to $\mathcal{S}_c(X)$. Nevertheless, a natural question is whether the super calibers of $\mathcal{S}_c(X)$ are, in turn, super calibers for $X$. In the following proposition, we show that the super calibers of $X$ do transfer to the subspace $L_X$:

\begin{proposition}\label{prop_CF_LX} If $X$ is a Hausdorff space, then $\SC(\mathcal{S}_c(X)) \subseteq \SC(L_X)$.

\end{proposition}

\begin{proof} Let $\kappa \in \SC(\mathcal{S}_c(X))$, $\{x_\alpha : \alpha<\kappa\} \subseteq L_X$, and $\{U_\alpha : \alpha<\kappa\} \subseteq \tau_{L_X}$ be such that $x_\alpha \in U_\alpha$ for every $\alpha<\kappa$. For each $\alpha<\kappa$, there exists $V_\alpha \in \tau_X$ with $U_\alpha = V_\alpha \cap L_X$, and, since $x_\alpha \in V_\alpha$, there is also some $S_\alpha \in \mathcal{S}_c(X,x_\alpha)$ such that $S_\alpha \subseteq V_\alpha$; equivalently, $S_\alpha \in \langle V_\alpha \rangle \cap \mathcal{S}_c(X)$ for every $\alpha<\kappa$. Thus, there exists $I \in [\kappa]^{\kappa}$ with $S_\alpha \in \langle V_\beta \rangle \cap \mathcal{S}_c(X)$ for all $\alpha,\beta \in I$, and therefore $x_\alpha \in U_\beta$ whenever $\alpha,\beta \in I$. Hence, $\kappa \in \SC(L_X)$.
\end{proof}

Recall that a topological space $X$ is \emph{sequential} if, for every $A \subseteq X$, the condition $A \notin \CL(X)$ implies that there exist a sequence $S \subseteq A$ and a point $x \in X \setminus A$ such that $S \to x$.

Every sequential space without isolated points has the property that $L_X = X$. Therefore, for this class of spaces, Proposition~\ref{prop_CF_LX} yields the following consequence:

\begin{corollary}\label{cor_CF_LX} If $X$ is a Hausdorff, sequential space without isolated points, then $\SC(\mathcal{S}_c(X)) \subseteq \SC(X)$.

\end{corollary}

Regarding the natural question of whether, in general, the super calibers of the hyperspace $\mathcal{S}_c(X)$ transfer to the space $X$, several considerations are relevant. Suppose that we seek a Hausdorff space $X$ such that
\[
\omega_1 \in \SC(\mathcal{S}_c(X)) \setminus \SC(X).
\]

A first remark is that $\mathcal{S}_c(X)$ cannot have a countable network, since otherwise $X$ would also have a countable network (see \cite[Theorem~8.3, p.~439]{mpp2018}), which would imply that $\omega_1$ is also a super caliber for $X$, by Proposition~\ref{prop_CF_basico}(4).

Second, if such an example exists in \textsf{ZFC}, then $X$ cannot be $T_3$. Indeed, if $X$ is $T_3$, then $\mathcal{S}_c(X)$ is also $T_3$ (since $\K(X)$ is), and consequently, under $\stick_{\text{SUPER}}$, it must have a countable network (see Proposition~\ref{prop_CF_basico}(9)), which is impossible.

Finally, the hyperspace $\mathcal{S}_c(X)$ would provide, in \textsf{ZFC}, an example of a space with super caliber $\omega_1$ and without a countable network. To the best of the author's knowledge, there is only one example with these properties in the literature, constructed in \cite{dowjuh2025} using the technique of elementary submodels. In light of these observations, strengthening Proposition~\ref{prop_CF_LX} and Corollary~\ref{cor_CF_LX} appears to be more difficult than usual. We leave the following question open:

\begin{question} Does there exist a Hausdorff space $X$ such that
\[
\omega_1 \in \SC(\mathcal{S}_c(X)) \setminus \SC(X)?
\]

\end{question}

%%%%%%%%%%%%%%%%%%%%%%%%%%%%%%%%%%%%%%%%%%%%%%%%%%%%%%%%%%%%%%
\section{Topological hyperspaces over metrizable spaces}\label{secc_THM}
%%%%%%%%%%%%%%%%%%%%%%%%%%%%%%%%%%%%%%%%%%%%%%%%%%%%%%%%%%%%%%

In contrast to Example~\ref{ej_CF_no_sube_Sc(X)}, in the realm of metrizable spaces, it is possible to obtain a strengthening of Corollary~\ref{cor_CF_F(X)}; however, some preliminaries are needed for this purpose.

A \emph{$k$-network} on a topological space $X$ is a family of subsets $\mathcal{N}$ of $X$ with the following property: whenever $K \in \mathcal{K}(X)$ and $U \in \tau_X$ satisfy $K \subseteq U$, there exists $\mathcal{F} \in [\mathcal{N}]^{<\omega}$ such that $K \subseteq \bigcup \mathcal{F} \subseteq U$. Since every base for $X$ is a $k$-network and every $k$-network is a net, it follows immediately that, if
$$
knw(X) := \min\{|\mathcal{N}| : \mathcal{N} \text{ is a $k$-network for } X\} + \omega,
$$
then $nw(X) \leq knw(X) \leq w(X)$.

In \cite[Theorem~1.1, p.~361]{ntantu1985}, it is shown that if $X$ is a Tychonoff space, then $nw(\K(X)) = knw(X)$. In particular, if $X$ is a Tychonoff space satisfying $nw(X) = w(X)$ (for example, any metrizable space or any $p$-space), then $nw(\K(X)) = nw(X)$.

With this in mind, and under the convention that $\mathcal{F}_1(X) := \{\{x\} : x \in X\}$, we are in a position to prove the following result:

\begin{proposition}\label{prop_CF_metrizable_hip} The following statements hold for every infinite metrizable space $X$.

\begin{enumerate}
\item If $\mathcal{F}_1(X) \subseteq \HH(X) \subseteq \K(X)$, then $\SC(X) = \SC(\HH(X))$.
\item If $\mathcal{S}_c(X) \neq \emptyset$ and $\mathcal{S}_c(X) \subseteq \HH(X) \subseteq \K(X)$, then $\SC(X) = \SC(\HH(X))$.

\end{enumerate}

\end{proposition}

\begin{proof} First, observe that, since the hyperspace $\K(X)$ is metrizable whenever $X$ is metrizable, the subspace $\HH(X)$ is also metrizable. Moreover, since $nw(\K(X)) = nw(X)$, Proposition~\ref{prop_CF_metrizable} implies that $\SC(X) = \SC(\K(X))$.

Now, if $\mathcal{F}_1(X) \subseteq \HH(X)$, then, by Proposition~\ref{prop_CF_basico}(1), the equality $\SC(X) = \SC(\K(X))$, and the fact that $X$ is homeomorphic to $\mathcal{F}_1(X)$, it follows that $\SC(X) = \SC(\HH(X))$.

On the other hand, if $\mathcal{S}_c(X) \neq \emptyset$ (i.e., if $X$ is not discrete), then $nw(\mathcal{S}_c(X)) = nw(X)$ (see \cite[Proposition~8.8, p.~441]{mpp2018}). Thus, from $\mathcal{S}_c(X) \subseteq \HH(X) \subseteq \K(X)$, it follows that $nw(\HH(X)) = nw(X)$. Finally, Proposition~\ref{prop_CF_metrizable} guarantees that $\SC(X) = \SC(\HH(X))$.
\end{proof}

In particular, Proposition~\ref{prop_CF_metrizable_hip} implies that $\SC(X)=\SC(\K(X))$ whenever $X$ is an infinite metrizable space. In light of this, a natural question arises: does $\SC(X)=\SC(\CL(X))$ hold under the same hypothesis?

It turns out that, in general, this is no longer the case, and the answer may even be independent of the underlying axiomatic system (see Theorems~\ref{thm_CF_metrizable_1}, \ref{thm_CF_metrizable_2}, and \ref{thm_CF_metrizable_3}). The following results provide intermediate steps in this direction.

\begin{lemma}\label{lema_CF_metrizable_1} If $X$ is a topological space and $A$ is a closed discrete subspace of $X$, then
\[
\SC(\CL(X)) \subseteq \big\{\kappa \in \CN : \cf(\kappa)>2^{|A|}\big\}.
\]

\end{lemma}

\begin{proof} For each $x \in A$, let $U_x \in \tau_X$ be such that $U_x \cap A = \{x\}$. Moreover, for every $S \subseteq A$, let $U_S := \bigcup_{x \in S} U_x$. Clearly, $\mathcal{P}(A) \subseteq \CL(X)$, and if $S,T \subseteq A$ are distinct, then $S \not\subseteq U_T$ or $T \not\subseteq U_S$; equivalently, $S \not\in \vt{U_T}$ or $T \not\in \vt{U_S}$.

Therefore, no cardinal less than or equal to $|\mathcal{P}(A)| = 2^{|A|}$ can be a super caliber for $\CL(X)$. Consequently, by Proposition~\ref{prop_CF_basico}(2), if $\kappa \in \SC(\CL(X))$, then $\cf(\kappa) > 2^{|A|}$.
\end{proof}

The first statement of Lemma~\ref{lema_CF_metrizable_2} was proved in the proof of part \textsc{II} of \cite[Theorem~8.5, p.~34]{hodel1984}. The second statement is established by a similar technique, while the third is an immediate consequence of the second.

\begin{lemma}\label{lema_CF_metrizable_2} Let $X$ be a metrizable space, $A$ a discrete subspace of $X$, $\kappa$ an infinite cardinal, and $\lambda$ a cardinal.

\begin{enumerate}
\item If $\cf(|A|)>\omega$, then there exists $B \subseteq A$ such that $|B| = |A|$ and $B$ has no accumulation points.
\item If $|A| > \kappa$, then there exists $B \subseteq A$ such that $|B| > \kappa$ and $B$ has no accumulation points.
\item If $|A| > \kappa\geq \lambda$, then there exists $B \subseteq A$ such that $|B| = \lambda$ and $B$ has no accumulation points.

\end{enumerate}

\end{lemma}

For Theorem~\ref{thm_CF_metrizable_1}, it is useful to keep in mind that, if $X$ is an infinite metrizable space, then $w(X) = nw(X)$ and $X$ contains discrete subspaces of cardinality $w(X)$ (see \cite{hodel1984}). It is also important to recall that $\SC(X)=\{\kappa \in \CN : \cf(\kappa)>w(X)\}$, by Proposition~\ref{prop_CF_metrizable}.

\begin{theorem}\label{thm_CF_metrizable_1} The following statements hold for every infinite metrizable space $X$.

\begin{enumerate}
\item If $\cf(w(X))>\omega$, then \[
w(X)^+ \in \SC(X) \setminus \SC(\CL(X)).
\]
\item If $w(X) = \omega$, then $X$ is compact if and only if $\SC(X) = \SC(\CL(X))$.

\end{enumerate}

\end{theorem}

\begin{proof} Regarding (1), the remark preceding this theorem, Lemma~\ref{lema_CF_metrizable_1}, and part (1) of Lemma~\ref{lema_CF_metrizable_2} imply that $w(X)^+ \in \SC(X)\setminus \SC(\CL(X))$.

As for (2), if $X$ is compact, then $\CL(X) = \K(X)$, and Proposition~\ref{prop_CF_metrizable_hip} ensures that $\SC(X) = \SC(\K(X)) = \SC(\CL(X))$. Conversely, if $X$ is not compact, then it is not countably compact, which implies that $\omega_1 \in \SC(X) \setminus \SC(\CL(X))$, by Lemma~\ref{lema_CF_metrizable_1}, since, under this hypothesis, $X$ contains a countably infinite closed discrete subspace.
\end{proof}

According to the statements of Theorem~\ref{thm_CF_metrizable_1}, it remains to analyze what happens when $X$ is an infinite metrizable space, $w(X)>\omega$, and $\cf(w(X))=\omega$. This is the purpose of the remainder of this section.

In Example~\ref{ej_CF_metrizable_X_kappa}, Theorem~\ref{thm_CF_metrizable_2}, and Lemmas~\ref{lema_CF_metrizable_3},~\ref{lema_CF_metrizable_4}, and~\ref{lema_CF_metrizable_5}, $\kappa$ denotes an uncountable cardinal such that $\cf(\kappa)=\omega$. Moreover, $\{\kappa_n : n < \omega\}$ is a strictly increasing sequence of cardinals satisfying $\kappa_0 = 0$, $\kappa_1 = \omega$, and $\kappa = \sup\{\kappa_n : n < \omega\}$.

It is also useful to recall the standard notation for intervals of ordinals: if $\alpha, \beta \in \ON$ satisfy $\alpha < \beta$, then $[\alpha, \beta) := \{\gamma \in \ON : \alpha \leq \gamma < \beta\}$ and $[\alpha, \beta] := [\alpha, \beta) \cup \{\beta\}$.

\begin{example}\label{ej_CF_metrizable_X_kappa} Let $I_n := [\kappa_n,\kappa_{n+1})$ for each $n < \omega$. Also, let $X_\kappa := \kappa + 1$ and, for each $\alpha \in X_\kappa$, let
\[
\ell(\alpha) := \begin{cases} n, & \text{if } \exists \, n<\omega \, (\alpha \in I_n), \\
\omega, & \text{if } \alpha = \kappa.
  \end{cases}
\]
Finally, let $d : X_\kappa \times X_\kappa \to \mathbb{R}$ be the function given by the following rule:
\[
d(\alpha, \beta) :=
\begin{cases} 
0, & \text{if } \alpha = \beta,\\[1mm]
2^{-\min\{\ell(\alpha),\ell(\beta)\}}, & \text{if } \alpha \neq \beta.
\end{cases}
\]

Under these circumstances, a standard argument shows that $d$ is an ultrametric on $X_\kappa$. Moreover, if $B(x,\varepsilon)$ denotes the open ball centered at $x \in X_\kappa$ with radius $\varepsilon>0$ (that is, $B(x,\varepsilon) := \{y \in X_\kappa : d(x,y)<\varepsilon\}$), then the following holds:

\begin{itemize}
\item If $\alpha \in X_\kappa \setminus \{\kappa\}$ and $n<\omega$ satisfies $\alpha \in I_n$, then 
$B(\alpha,2^{-(n+1)}) = \{\alpha\}$.

\item On the other hand, if $0<\varepsilon \leq 1$ and 
$m_\varepsilon := \max\{m<\omega : 2^{-m} \geq \varepsilon\}$, then 
$B(\kappa,\varepsilon) = \{\kappa\} \cup \bigcup_{m>m_\varepsilon} I_m$. In particular, if $n<\omega$, then 
\[B(\kappa,2^{-n}) = \{\kappa\} \cup \bigcup_{m>n} I_m = \{\kappa\} \cup \bigcup_{m>n} [\kappa_m,\kappa_{m+1}) = [\kappa_{n+1},\kappa].\]

\end{itemize}

The ultrametrizable space $X_\kappa$ satisfies $nw(X_\kappa) = \kappa = w(X_\kappa)$. Indeed, the family
\[
\mathcal{B}_\kappa := \{\{\alpha\} : \alpha < \kappa\} \cup \{[\kappa_n,\kappa] : n < \omega\}
\]
forms a basis for $X_\kappa$ of cardinality $\kappa$. Moreover, $X_\kappa$ contains $\kappa$ isolated points, so its net weight cannot be less than $\kappa$.

Additionally, if \[\mathcal{A}(\kappa) := \{A\cup \{\kappa\} : A \in \Pot(\kappa)\} \quad \text{and} \quad \mathcal{B}(\kappa) := \{A \subseteq \kappa : \text{$A$ is bounded in $\kappa$}\},
\] then $\CL(X_\kappa) = \mathcal{A}(\kappa) \cup \mathcal{B}(\kappa)$.

\end{example}

For the following result, it is useful to keep in mind the following basic fact about cardinal functions on hyperspaces: if $X$ is an infinite discrete space, then $nw(\CL(X)) = 2^{|X|}$.

\begin{lemma}\label{lema_peso_red_X_kappa} The space $X_\kappa$ from Example~\ref{ej_CF_metrizable_X_kappa} satisfies
\[
nw(\CL(X_\kappa)) = 2^{<\kappa} = w(\CL(X_\kappa)).
\]

\end{lemma}

\begin{proof} First, for each $n<\omega$, the subspace $Y_n := \{A \in \CL(X_\kappa) : A \subseteq \kappa_n\}$ is homeomorphic to the hyperspace $\CL(\kappa_n)$, where $\kappa_n$ is equipped with the discrete topology. Moreover, since $\sup\{\kappa_n : n < \omega\} = \kappa$, it follows that
\begin{align*}
2^{<\kappa} = \sup\{2^{\kappa_n} : n<\omega\} &= \sup\{nw(\CL(\kappa_n)) : n <\omega\} \\
&= \sup\{nw(Y_n) : n < \omega\} \leq nw(\CL(X_\kappa)).
\end{align*}

On the other hand, let
\[
\textstyle
\mathcal{B} := \bigcup_{n<\omega} \big(\Pot(\kappa_n) \cup \{A \cup [\kappa_n,\kappa] : A \in \Pot(\kappa_n)\}\big).
\]
Then $\mathcal{B}$ is a base for $X_\kappa$, which, in particular, implies that the collection
\[
\textstyle
\mathcal{V} := \{\langle B_0, \ldots, B_m\rangle : m<\omega \ \wedge \ \{B_i : i \leq m\} \subseteq \mathcal{B}\}
\]
consists of open subsets of $\CL(X_\kappa)$. Our goal is to verify that $\mathcal{V}$ is a base for $\CL(X_\kappa)$.

Let $O$ be an open subset of $\CL(X_\kappa)$, let $A$ be an element of $O$, let $m\in \mathbb{N}$, and let $U_1,\ldots,U_m \in \tau_{X_\kappa}$ be such that $A \in \langle U_1,\ldots,U_m\rangle \subseteq O$. For each $1 \leq i \leq m$, let $\alpha_i \in A \cap U_i$ and let $B_i \in \mathcal{B}$ be such that $\alpha_i \in B_i \subseteq U_i$. Moreover, if $\kappa \not \in A$, let $B_0 := A$; otherwise, the condition $\kappa \in A$ implies that there is some $n<\omega$ such that $[\kappa_n,\kappa] \subseteq \bigcup_{i=1}^{m} U_i$, and we let $B_0 := (A\cap \kappa_n) \cup [\kappa_n,\kappa]$. Clearly, $\langle B_0,\ldots,B_m \rangle \in \mathcal{V}$, $A \in \langle B_0,\ldots,B_m \rangle$, and $\langle B_0,\ldots,B_m \rangle \subseteq \langle U_1,\ldots,U_m\rangle \subseteq O$.

In summary, since $\mathcal{V}$ is a base for $\CL(X_\kappa)$, it follows that
\[
nw(\CL(X_\kappa)) \leq w(\CL(X_\kappa)) \leq |\mathcal{V}| \leq |[\mathcal{B}]^{<\omega}| = |\mathcal{B}| \leq \omega \cdot \sup \{2^{\kappa_n} : n <\omega\} = 2^{<\kappa}.
\]
and, therefore,
\[
nw(\CL(X_\kappa)) = 2^{<\kappa} = w(\CL(X_\kappa)).
\]
\end{proof}

Lemma~\ref{lema_peso_red_X_kappa} and the fact that, under the Generalized Continuum Hypothesis, \textsf{GCH}, we have $2^{<\kappa} = \kappa$, are all we need to prove Theorem~\ref{thm_CF_metrizable_2}:

\begin{theorem}\label{thm_CF_metrizable_2} If \textsf{GCH} holds, then the space $X_\kappa$ from Example~\ref{ej_CF_metrizable_X_kappa} satisfies
\[
\SC(X_\kappa) = \SC(\CL(X_\kappa)).
\]

\end{theorem}

\begin{proof} First, since $nw(X_\kappa) = \kappa$, Proposition~\ref{prop_CF_metrizable} implies that $\SC(X_\kappa) = \{\lambda \in \CN : \cf(\lambda)>\kappa\}$. For the reverse inclusion, since $X_\kappa$ is homeomorphic to a subspace of $\CL(X_\kappa)$ and $nw(\CL(X_\kappa)) = \kappa$, items (1) and (4) of Proposition~\ref{prop_CF_basico} imply that $\{\lambda \in \CN : \cf(\lambda)>\kappa\} \subseteq \SC(\CL(X_\kappa)) \subseteq \SC(X_\kappa)$. Hence, $\SC(X_\kappa) = \SC(\CL(X_\kappa))$.
\end{proof}

It is an appropriate time to pause and discuss a little bit of cardinal arithmetic. Under our hypotheses, it is straightforward to verify that 
$|\mathcal{B}(\kappa)| = \sup\{2^{\kappa_n} : n<\omega\}$. Therefore, since
\[
\mathfrak{c} = |\mathcal{P}(\omega)| = |\mathcal{P}([\kappa_0,\kappa_1))| \leq |\mathcal{B}(\kappa)|,
\]
the hypothesis $\mathfrak{c} > \kappa$, which is consistent with \textsf{ZFC}, implies that $|\mathcal{B}(\kappa)| \geq \kappa^+$.

The following lemmas are intended to pave the way for the proof of Theorem~\ref{thm_CF_metrizable_3}, in which we obtain a consistency result that depends on the size of the continuum $\mathfrak{c}$.

\begin{lemma}\label{lema_CF_metrizable_3} The inequality $\mathfrak{c} > \kappa$ implies that the space $X_\kappa$ from Example~\ref{ej_CF_metrizable_X_kappa} satisfies $\kappa^+ \not\in \SC(\CL(X_\kappa))$.

\end{lemma}

\begin{proof} Let $\{A_\alpha : \alpha<\kappa^+\}$ be a subset of $\mathcal{B}(\kappa)$ enumerated without repetitions. Observe that, for every $\alpha<\kappa^+$, we have $A_\alpha \in \langle A_\alpha \rangle$. Moreover, if $\alpha<\beta<\kappa^+$, it cannot be the case that both $A_\alpha \in \langle A_\beta \rangle$ and $A_\beta \in \langle A_\alpha \rangle$, since this would imply that $A_\alpha = A_\beta$. Consequently, $\kappa^+ \not\in \SC(\CL(X_\kappa))$.
\end{proof}

\begin{lemma}\label{lema_CF_metrizable_4} Let $Z$ be a topological space, let $C$ be a discrete subspace of $Z$, and let $p \in Z \setminus C$ be such that $Z = C \cup \{p\}$. If there is an increasing sequence $\{C_n : n \in \mathbb{N}\}$ such that $C = \bigcup_{n \in \mathbb{N}} C_n$, $|C_n| = \kappa_n$ for every $n \in \mathbb{N}$, and $\{Z \setminus C_n : n \in \mathbb{N}\}$ is a local base for $p$ in $Z$, then $Z$ is homeomorphic to the space $X_\kappa$ from Example~\ref{ej_CF_metrizable_X_kappa}.

\end{lemma}

\begin{proof} Let $C_0 := \emptyset$ and let $f_{n+1} : [\kappa_n,\kappa_{n+1}) \to C_{n+1} \setminus C_n$ be a bijection for every $n < \omega$. Clearly,
\[\textstyle
f := \{(\kappa,p)\} \cup \bigcup\{f_n : n \in \mathbb{N}\}
\]
is a bijection between $X_\kappa$ and $Z$, with the property that the restriction $\restr{f}{\kappa} : \kappa \to C$ is, in turn, a bijection between the isolated points of $X_\kappa$ and those of $Z$. Moreover, for every $n < \omega$, we have $[\kappa_n,\kappa] = f^{-1}[Z \setminus C_n]$. Therefore, since
\[
\mathcal{B}_\kappa := \{\{\alpha\} : \alpha < \kappa\} \cup \{[\kappa_n,\kappa] : n < \omega\}
\ \text{and} \
\mathcal{B} := \{\{x\} : x \in C\} \cup \{Z \setminus C_n : n < \omega\}
\]
are bases for $X_\kappa$ and $Z$, respectively, it follows that $f : X_\kappa \to Z$ is a homeomorphism.
\end{proof}

For Lemma~\ref{lema_CF_metrizable_5}, it is useful to recall an important concept: if $X$ is a topological space, $A$ is a subset of $X$, and $p \in X$, then $p$ is a \emph{complete accumulation point} of $A$ in $X$ if $|A \cap U| = |A|$ for every $U \in \tau_X$ such that $p \in U$.

\begin{lemma}\label{lema_CF_metrizable_5} Let $Y$ be a first countable space, let $A$ be a discrete subspace of $Y$, and let $p \in Y \setminus A$ be such that $Y = A \cup \{p\}$ and $|A| = \kappa$. If $Y$ has no subspaces of cardinality $\kappa$ without accumulation points, then $Y$ contains a closed subspace homeomorphic to the space $X_\kappa$ from Example~\ref{ej_CF_metrizable_X_kappa}.

\end{lemma}

\begin{proof} First, observe that $p$ is a complete accumulation point of $A$ in $Y$. Indeed, suppose otherwise; that is, suppose that there exists $U \in \tau_Y$ such that $p \in U$ and $|A \cap U| < \kappa$. Then, since $p \not\in \overline{A \setminus U}$, it follows that $A \setminus U$ is a closed and discrete subspace of $Y$, which implies that $|A \setminus U| < \kappa$ and, consequently, that $|A| < \kappa$, a contradiction.

Let $\{B_n : n \in \mathbb{N}\}$ be a decreasing local base for $p$ in $Y$. We now proceed with a recursive construction:

\medskip

\noindent {\bf Claim.} There exists a collection $\{A_n : n \in \mathbb{N}\}$ of closed and discrete subspaces of $Y$ such that
\[
\textstyle
|A_n| = \kappa_n \quad \text{and} \quad A_n \subseteq (A \cap B_n) \setminus \bigcup_{0 < m < n} A_m
\]
for every $n \in \mathbb{N}$.

\medskip

Indeed, suppose that, for some $n \in \mathbb{N}$, the family $\{A_m : 0 < m < n\}$ has been constructed with the desired properties. The relations $|A \cap B_n| = \kappa$ and $|\bigcup_{0 < m < n} A_m| < \kappa$ imply that $(A \cap B_n) \setminus \bigcup_{0 < m < n} A_m$ is a discrete subspace of $Y$ of cardinality $\kappa$. By Lemma~\ref{lema_CF_metrizable_2}(3), there exists $A_n \subseteq (A \cap B_n) \setminus \bigcup_{0 < m < n} A_m$ such that $|A_n| = \kappa_n$ and $A_n$ has no accumulation points in $Y$. Clearly, $\{A_m : 0 < m \leq n\}$ has the required properties.

\medskip

Let $C_n := \bigcup_{0 < m \leq n} A_m$ for each $n \in \mathbb{N}$. Observe that, if $C := \bigcup_{n \in \mathbb{N}} C_n$ and $Z := C \cup \{p\}$, then $Z$ is a closed subspace of $Y$, $C$ is a discrete subspace of $Z$, the family $\{C_n : n \in \mathbb{N}\}$ is increasing, and $|C_n| = \kappa_n$ for every $n \in \mathbb{N}$.

To verify, by Lemma~\ref{lema_CF_metrizable_4}, that $Z$ is homeomorphic to the space $X_\kappa$ from Example~\ref{ej_CF_metrizable_X_kappa}, it remains to show that $\{Z \setminus C_n : n \in \mathbb{N}\}$ is a local base for $p$ in $Z$. Since $\{Z \cap B_n : n \in \mathbb{N}\}$ is a local base for $p$ in $Z$, it suffices to show that $Z \setminus C_n \subseteq Z \cap B_{n+1}$ for every $n \in \mathbb{N}$, since $p \in \bigcap_{n \in \mathbb{N}} (Z \setminus C_n)$. To this end, note that, if $n \in \mathbb{N}$, then
\begin{align*}
Z \setminus (C_n \cup \{p\}) \subseteq C \setminus C_n &= \textstyle \bigcup_{m \in \mathbb{N}} A_m \setminus \bigcup_{0 < m \leq n} A_m \\
&= \textstyle \bigcup_{m > n} A_m \subseteq \bigcup_{m > n} B_m \subseteq B_{n+1}.
\end{align*}
\end{proof}

\begin{theorem}\label{thm_CF_metrizable_3} Let $X$ be an infinite metrizable space. If $w(X)>\omega$, $\cf(w(X))=\omega$ and $\mathfrak{c} > w(X)$, then
\[
w(X)^+ \in \SC(X) \setminus \SC(\CL(X)).
\]

\end{theorem}

\begin{proof} Let $\kappa := w(X)$. First, if $X$ contains a closed discrete subspace of cardinality $\kappa$, then Lemma~\ref{lema_CF_metrizable_1} directly implies that $\kappa^+ \in \SC(X)\setminus \SC(\CL(X))$. Therefore, we may assume that $X$ contains no subspaces of cardinality $\kappa$ without accumulation points.

Under this assumption, according to \textsc{II} in the proof of \cite[Theorem~8.5, p.~34]{hodel1984}, the space $X$ contains a discrete subspace of cardinality $\kappa$ with exactly one accumulation point. Let $A$ be a discrete subspace of $X$ and let $p \in X \setminus A$ be such that $|A|=\kappa$ and $\overline{A}=A\cup\{p\}$. Observe that, since the subspace $Y := A \cup \{p\}$ is closed in $X$, it has no subspaces of cardinality $\kappa$ without accumulation points. Thus, Lemma~\ref{lema_CF_metrizable_5} guarantees that $Y$ contains a closed subspace $Z$ homeomorphic to the space $X_\kappa$ from Example~\ref{ej_CF_metrizable_X_kappa}.

Finally, since $Z$ is closed in $Y$ and $Y$ is closed in $X$, it follows that $\CL\vt{Z}$ is a subspace of $\CL(X)$, and hence Proposition~\ref{prop_CF_basico}(1) implies that
\[
\SC(\CL(X)) \subseteq \SC(\CL\vt{Z}).
\]
Consequently, since Lemma~\ref{lema_CF_metrizable_3} establishes that $\kappa^+ \not\in \SC(\CL(X_\kappa))$, we conclude that
\[
\kappa^+ \in \SC(X) \setminus \SC(\CL(X)).
\]
\end{proof}

In summary, the question of whether an infinite metrizable space $X$ satisfies $\SC(X)=\SC(\CL(X))$ is settled in \textsf{ZFC} when $w(X) = \omega$ or $\cf(w(X))>\omega$, whereas it is independent of \textsf{ZFC} when $w(X) > \omega$ and $\cf(w(X)) = \omega$. We conclude the present paper with a single result that collects the results established in Theorems~\ref{thm_CF_metrizable_1}, \ref{thm_CF_metrizable_2}, and \ref{thm_CF_metrizable_3}.

\begin{theorem}\label{thm_CF_principal} The following statements hold for every infinite metrizable space $X$.

\begin{enumerate}
\item If $w(X) = \omega$, then $X$ is compact if and only if $\SC(X) = \SC(\CL(X))$.

\item If $\cf(w(X))>\omega$, then
\[
w(X)^+ \in \SC(X) \setminus \SC(\CL(X)).
\]

\item If $w(X)>\omega$, $\cf(w(X))=\omega$, and the inequality $\mathfrak{c} > w(X)$ holds, then
\[
w(X)^+ \in \SC(X) \setminus \SC(\CL(X)).
\]

\item There exists an ultrametrizable space $Y$ such that $w(Y)>\omega$, $\cf(w(Y))=\omega$, and, if \textsf{GCH} holds, then
\[
\SC(Y) = \SC(\CL(Y)).
\]

\end{enumerate}

\end{theorem}

\end{document}